\documentclass[12pt]{article} 
\usepackage{amsmath, amssymb}
\usepackage{amsthm} 
\theoremstyle{plain} 
\newtheorem{theorem}{\indent\sc Theorem}[section]

\newtheorem{corollary}[theorem]{\indent\sc Corollary}
\newtheorem{proposition}[theorem]{\indent\sc Proposition}

\theoremstyle{definition} 

\newtheorem{remark}[theorem]{\indent\sc Remark}

\newtheorem{ex}[theorem]{\indent\sc Example}

\makeatletter
\def\address#1#2{\begingroup
\noindent\parbox[t]{7.8cm}{%
\small{\scshape\ignorespaces#1}\par\vskip1ex
\noindent\small{\itshape E-mail address}%
\/: #2\par\vskip4ex}\hfill%
\endgroup}%
\makeatother
\title{\uppercase {The Bi-normal fields on spacelike surfaces in $\mathbb R_1^4$ }} 
\author{ Dang Van Cuong
\bigskip
\textsc{ }\footnote{Hue Geometry Group}}
\date{\today}
\begin{document}
\maketitle
\footnote{AMS
\textit{Mathematics Subject Classification}.
53C50, 53A35, 57R45.
}

\begin{abstract}
A normal field on a spacelike surface in $R^4_1$ is called bi-normal if $K^{\nu}$, the determinant of Weingarten map associated with $\nu$,  is zero. In this paper  we give a relationship between the spacelike pseudo-planar surfaces and spacelike pseudo-umbilical surfaces,  then  study the bi-normal fields on spacelike ruled surfaces and spacelike surfaces of revolution.
\end{abstract}
\setcounter{section}{-1}
\section{Introduction}
Let $\alpha:I\to \mathbb R^3$ be a bi-regular parametric curve. A long this curve, the vector field defined by
$$\textbf b=\frac{1}{|\alpha'\times \alpha''|}(\alpha'\times \alpha'') $$
is called the bi-normal field of $\alpha.$ A bi-normal vector can be seen as a direction whose the corresponding height function has a degenerate (non-Morse) critical point.\\
\indent Let $M$ be a regular surface in $\mathbb R_1^4$ (or $\mathbb R^4$) and $f_{\textbf v}$ be the height function on $M$ associated with a direction $\textbf v$. By analogy with the case of curves in $\mathbb R^3$, a direction $\textbf v$ is called a bi-normal direction  of $M$ at a point $p$ if the height function $f_{\textbf v}$ has a degenerate singularity at $p.$ The height function $f_{\textbf v}$ having a degenerate singularity means that its hessian is singular.\\
 \indent Given a normal field $\nu$ on $M$ and denoted by $S^{\nu}$ the shape operator associated with $\nu.$ The $\nu$-Gauss curvature $K^{\nu}$ of $M$ is defined by $K^{\nu}=\det S^{\nu}.$  The eigenvalues $k_1^{\nu}$ and $k_2^{\nu}$ of the shape operator $S^{\nu}$ are called $\nu$-principal curvatures. The $\nu$-mean curvature of $M$ is defined by
 $$H^{\nu}=\frac{1}{2}{\rm trace}S^{\nu}=\frac{1}{2}(k_1^{\nu}+k_2^{\nu}).$$
 A point $p\in M$ is called $\nu$-umbilic if $k_1^{\nu}(p) = k_2^{\nu}(p) = k$ and is called $\nu$-flat if $k_1^{\nu}(p) = k_2^{\nu}(p) = 0.$  If there exists a normal  field  $\nu$ on $M$ such that $M$ is $\nu$-umbilic ($\nu$-flat) then $M$ is called pseudo-umbilical (pseudo-flat) surface. $M$ is called umbilic if it is $\nu$-umbilic for all normal fields $\nu.$ $M$ is called maximal if $H^{\nu}=0$ for all normal fields $\nu.$\\
  \indent It is easy to show that $\nu(p)$ is bi-normal direction of $M$ at $p$ if $\det S^{\nu(p)}=0 $ i.e. either $k_1^{\nu}(p)=0$ or $k_2^{\nu}(p)=0.$ Such a point is called $\nu$-planar and a direction belonged to the kernel of $S^{\nu(p)}$ is said to be asymptotic. The normal field $\nu$ of $M$ is called bi-normal field if for each $p\in M,$ $\nu(p)$ is bi-normal direction of $M$ at $p.$ If there exists a bi-normal field on $M$ then $M$ is called pseudo-planar surface, in the case each normal field is bi-normal $M$ is called planar surface.  For everything concerning to these notions in more detail, we refer the reader to \cite{izu1}, \cite{izu3}, \cite{F1},  \cite{F2}, \cite{F3} and references therein.\\
\indent The existence of bi-normal direction on  surfaces in $\mathbb R^4$ has been studied by several authors (e.g. \cite{drei}, \cite{F1}, \cite{F2}, \cite{F3}, \cite{F5}, \cite{lit}, \cite{W}, \dots). Little (\cite{lit}, Theorem 1.3(b), 1969) showed that a surface whose all normal fields are bi-normal  if and only if it is a ruled developable surface. In 1995, D.K.H. Mochida et. al (\cite{F1}, Corollary 4.3 repeated at \cite{F5},  (2010)) showed that a surface admitting two bi-normal fields if and only if it is strictly locally convex. These results was expanded to surfaces of codimension two in $\mathbb R^{n+2}$ by them \cite{F2} in 1999. These methods are used later by  M.C. Romero-Fuster and F. S\'{a}nchez-Brigas (\cite{F3}, Theorem 3.4,  2002) to study the umbilicity of surfaces.\\
\indent The first section of this paper  shows that there exist pseudo-planar surfaces are not pseudo-umbilic, defines the number of bi-normal fields on the pseudo-umbilical surfaces and gives some interesting corollaries. \\
\indent In the second of this paper we show that  each point on the spacelike ruled surfaces admits either one or all bi-normal directions, a spacelike ruled surface is pseudo-umbilic iff umbilic.\\
\indent In the third section of this paper we show that the spacelike surfaces of revolution admit exactly two bi-normal fields whose asymptotic fields respectively are orthogonal. Therefore, they are pseudo-umbilic but not umbilic.\\
\indent The final section of this paper shows that the number of bi-normal fields on the rotational spacelike surface whose meridians lie in two-dimension space are depended on the properties of its meridian.
\section{Bi-normal Fields on Pseudo-umbilical Surfaces}
For the surfaces in $\mathbb R^4$ Romero Fuster \cite{F3} showed that pseudo-umbilical surfaces are pseudo-planar; moreover, their two asymptotic fields are orthogonal. These results are also true for spacelike surfaces in $\mathbb R_1^4,$ and I would like to show it here. Notice that there exist the pseudo-planar spacelike surfaces are not pseudo-umbilic, let see Example \ref{ex1} and Example \ref{ex2}. We have the  similar example for surfaces in $\mathbb R^4.$ \\
\indent The following theorem shows that the pseudo-umbilical spacelike surfaces are pseudo-planar  and gives us the number of bi-normal fields on them.
\begin{theorem}\label{theoum} Let $M$ be a spacelike surface in $\mathbb R_1^4.$ If $M$ is pseudo-umbilic (not pseudo-flat)  then it admits either one or two bi-normal fields. Moreover, $M$ admits only one bi-normal field iff it is umbilic.
\end{theorem}
\begin{proof}
Suppose that $M$ is $\nu$-umbilic (not $\nu$-flat). Let ${\rm{\bf n}}$ be a normal field on $M$ such that $\{{\rm {\bf X}}_u,{\rm {\bf X}}_v,\nu,{\rm{\bf n}}\}$ is linearly independent and $k$ is $\nu$-principal curvature. Given a normal field  ${\rm{\bf B}},$ then we have the following interpretation
$${\rm{\bf B}}=\lambda \nu+\mu {\rm{\bf n}},$$
where $\lambda,\mu$ are smooth functions on $M.$ Suppose that the coefficients of the fist fundamental form of $M$ satisfy $$g_{11}=g_{22}=\varphi, g_{12}=0,$$ then we have
$$S^{\rm{\bf B}}=\lambda S^{\nu}+\mu S^{\rm{\bf n}}=\lambda k\left[\begin{matrix}1&0\\0&1\end{matrix}\right]+\frac{\mu}{\varphi} \left[\begin{matrix}b_{11}^{\rm {\bf n}}&b_{12}^{\rm {\bf n}}\\b_{12}^{\rm {\bf n}}&b_{22}^{\rm {\bf n}}\end{matrix}\right]= \left[\begin{matrix}\frac{\mu}{\varphi}  b_{11}^{\rm {\bf n}}+\lambda k&\frac{\mu}{\varphi}  b_{12}^{\rm {\bf n}}\\\frac{\mu}{\varphi}  b_{12}^{\rm {\bf n}}&\frac{\mu}{\varphi}  b_{22}^{\rm {\bf n}}+\lambda k\end{matrix}\right].$$
Therefore,
$$K^{\rm{\bf B}}=\gamma^2\left(b_{11}^{\rm {\bf n}}b_{22}^{\rm {\bf n}}-(b_{12}^{\rm {\bf n}})^2\right)+\lambda \gamma k\left(b_{11}^{\rm {\bf n}}+b_{22}^{\rm {\bf n}}\right)+\lambda^2k^2,  $$
\begin{equation}\label{umbinor}K^{\rm{\bf B}}=0\ \Leftrightarrow\ \gamma^2\left(b_{11}^{\rm {\bf n}}b_{22}^{\rm {\bf n}}-(b_{12}^{\rm {\bf n}})^2\right)+\lambda \gamma k\left(b_{11}^{\rm {\bf n}}+b_{22}^{\rm {\bf n}}\right)+\lambda^2k^2=0,\end{equation}
where $\gamma=\frac{\mu}{\varphi}.$
Since $\nu$ is not bi-normal,  $\mu\ne 0$. Then the equation (\ref{umbinor}) can be rewrote by
\begin{equation}\label{ptumbinor} \left(\frac{\lambda k}{\gamma} \right)^2+\left(b_{11}^{\rm {\bf n}}+b_{22}^{\rm {\bf n}}\right)\frac{\lambda k}{\gamma} +b_{11}^{\rm {\bf n}}b_{22}^{\rm {\bf n}}-(b_{12}^{\rm {\bf n}})^2=0. \end{equation}
It is from
\begin{equation}(b_{11}^{\rm {\bf n}}-b_{22}^{\rm {\bf n}})^2+4(b_{12}^{\rm {\bf n}})^2\geq0\end{equation}
that the equation (\ref{ptumbinor}) has at least one or at most two  solutions. That means $M$ admits at least one or  at most two bi-normal fields.\\
\indent $M$ admits only one bi-normal field if and only if $b_{11}^{\rm{\bf n}}=b_{22}^{\rm{\bf n}}$ and $b_{12}^{\rm{\bf n}}=0$. Which means that $M$ is ${\rm{\bf n}}$-umbilic. Then the Lemma 4.1 in \cite{C-H} shows that $M$ is umbilic.
\end{proof}
The following example gives a spacelike surface admitting one bi-normal field but not pseudo-umbilic.
\begin{ex}\label{ex1} Let $M$ be a surface given by following parameterization
\begin{equation}{\rm {\bf X}}(u,v)=\left(\cos u(1+v),\sin u(1+v),\sinh u,\cosh u\right),\ u,v\in\mathbb R.\end{equation}
The coefficients of the fist fundamental form of $M$ are determined by
$$g_{11}=\langle {\rm {\bf X}}_u,{\rm {\bf X}}_u  \rangle=v^2+2>0,\ g_{12}=\langle {\rm {\bf X}}_u,{\rm {\bf X}}_v  \rangle=0,\ g_{22}=\langle {\rm {\bf X}}_v,{\rm {\bf X}}_v  \rangle=1.$$
Therefore, $M$ is a spacelike surface.  Let ${\rm{\bf n}}=\left(n^1,n^2,n^3,n^4\right) $ be a normal field on $M.$ That means
\begin{equation}\label{tgXu}\langle {\rm {\bf X}}_u,{\rm{\bf n}}  \rangle=0\Leftrightarrow\ n^1\cos u+n^2\sin u=0,\end{equation}
\begin{equation}\label{tgXv}\langle {\rm {\bf X}}_v,{\rm{\bf n}}  \rangle=0\ \Leftrightarrow\ -n^1\sin u(1+v)+n^2\cos u(1+v)+n^3\cosh u-n^4\sinh u=0. \end{equation}
Using (\ref{tgXu}) the coefficients of the second fundamental form associated with ${\rm{\bf n}} $ are
$$b_{11}^{{\rm{\bf n}} }=n^3\sinh u-n^4\cosh u,\ b_{12}^{{\rm{\bf n}} }=-n_1\sin u+n_2\cos u,\ b_{22}^{{\rm{\bf n}} }=0.$$
We have
$$\det(b_{ij}^{\rm{\bf n}})=(b_{12}^{{\rm{\bf n}} })^2.$$
So,
\begin{equation}\label{Kn0}K^{{\rm{\bf n}} }=0\ \Leftrightarrow\ -n^1\sin u+n^2\cos u=0.\end{equation}
Connecting (\ref{tgXu}), (\ref{tgXv}) and (\ref{Kn0}) we imply that  ${\rm{\bf n}} $ is a bi-normal field on $M$ if and only if
$$n^1=n^2=0,\ n^3\cosh u-n^4\sinh u=0.$$
That means  $M$ admits only one bi-normal field $${\rm{\bf B}} =\left(0,0,\sinh u,\cosh u\right).$$
It is a unit timelike normal field. Since
$$S^{{\rm{\bf B}} }=\left[\begin{matrix}-1&0\\0&0\end{matrix}\right], $$
 $M$ is not ${\rm{\bf B}} $-flat.\\
On the other hand the ${\rm{\bf n}} $-principal curvatures are the solutions of the following equation
$$\det\left((b_{ij}^{{\rm{\bf n}} })-\lambda(g_{ij}) \right)=0\ \Leftrightarrow\ (v^2+2)\lambda^2-\lambda b_{11}^{{\rm{\bf n}} }-\left(b_{12}^{{\rm{\bf n}} }\right)^2=0.$$
Therefore, $M$ is ${\rm{\bf n}} $-umbilic if and only if $b_{11}^{{\rm{\bf n}} }=b_{12}^{{\rm{\bf n}} }=0.$ Which doesn't take place, by connecting (\ref{tgXu}), (\ref{tgXv}) and (\ref{Kn0}). So, $M$ is not pseudo-umbilic.
\end{ex}
Even when $M$ admits two bi-normal fields, it is not pseudo-umbilic. Let see the following example.
\begin{ex}\label{ex2} Let $M$ be a surface given by following parameterization
$${\rm {\bf X}}(u,v)=\left(e^{2u}\cos v,e^{2u}\sin v,e^{-u}\cosh v,e^{-u}\sinh v\right),\ u>1,\ v\in(0,2\pi).  $$
It is easy to show that $M$ is spacelike and $\left\{{\rm{\bf n}}_1,{\rm{\bf n}}_2  \right\} $ is a frame of the variable normal bundle, where
$${\rm{\bf n}}_1=-\frac{1}{\sqrt{g_{11}}}\left(e^{-u}\cos v,e^{-u}\sin v,2e^{2u}\cosh v, 2e^{2u}\sinh v\right),   $$
$${\rm{\bf n}}_2=\frac{1}{\sqrt{g_{22}}}\left(-e^{-u}\sin v,e^{-u}\cos v,e^{2u}\sinh v,e^{2u}\cosh v\right).$$
The coefficients of the second fundamental form associated with ${\rm{\bf n_1}}$ and ${\rm{\bf n_2}}$ are
$$b_{11}^{{\rm{\bf n}}_1}=-\frac{6e^u}{\sqrt{g_{11}}},\ b_{12}^{{\rm{\bf n}}_1 }=0,\ b_{22}^{{\rm{\bf n}}_1 }=-\frac{e^u}{\sqrt{g_{11}}};  $$
$$b_{11}^{{\rm{\bf n}}_2}=0,\ b_{12}^{{\rm{\bf n}}_2 }=\frac{3e^{u}}{\sqrt{g_{22}}} ,\ b_{22}^{{\rm{\bf n}}_2 }=0.  $$
 Therefore, both ${\rm{\bf n_1}}$ and ${\rm{\bf n_2}}$ are not bi-normal. Fore each normal field ${\rm{\bf n}} $ on $M$ we have the following interpretation
\begin{equation}\label{bieudienn}{\rm{\bf n}}={\rm{\bf n}}_1+\mu{\rm{\bf n}}_2\end{equation} and
$$\left(b_{ij}^{{\rm{\bf n}} }\right) =\left[\begin{matrix}b_{11}^{{\rm{\bf n}}_1 }&\mu b_{12}^{{\rm{\bf n}}_2 }\\ \mu b_{12}^{{\rm{\bf n}}_2 }&b_{22}^{{\rm{\bf n}}_1 }\end{matrix}\right].$$
So $$K^{{\rm{\bf n}} }=\frac{b_{11}^{{\rm{\bf n}}_1 }b_{22}^{{\rm{\bf n}}_1 }-\mu^{2}\left(b_{12}^{{\rm{\bf n}}_2 }\right)^2 }{g_{11}g_{22}}.$$
 Since $b_{11}^{{\rm{\bf n}}_1 }b_{22}^{{\rm{\bf n}}_1 }>0$ and $b_{12}^{{\rm{\bf n}}_2 }\ne0$, $M$ admits exactly two bi-normal fields.\\
\indent On the other hand the ${\rm{\bf n}}$-principal curvatures of $M$ are solutions of the following equation
\begin{equation}\label{gtr0rèn}\det\left((b_{ij}^{{\rm{\bf n}} })-\lambda (g_{ij})\right)=0\Leftrightarrow g_{11}g_{22})\lambda^2+\left(e^u\sqrt{g_{11}}+\frac{6e^ug_{22}}{\sqrt{g_{11}}} \right) \lambda +\frac{6e^{2u}}{g_{11}}-\frac{9e^{2u}\mu^2}{g_{22}}=0,\end{equation}
where $\lambda$ is the variable. Since
$$\frac{1}{g_{11}}\left[\left(2e^{4u}-7e^{-2u}\right)^2 \right]+36g_{11}\mu^2>0,\ \forall \mu,$$
for each normal field ${\rm{\bf n}},$ $M$ is not ${\rm{\bf n}}$-umbilic. It means that $M$ is not pseudo-umbilic.
\end{ex}
\begin{corollary} If $M$ is umbilic then $M$ is pseudo-flat.
\end{corollary}
\begin{corollary}
Let $M$ be a surface contained in the a pseudo-sphere (Hyperbolic  or de Sitter). Then the following statements are equivalent.
\begin{enumerate}
\item[(1)] $M$ is  umbilic,
\item[(2)] $M$ admits only one bi-normal field,
\item[(3)] $M$ is contained in a hyperplane.
\end{enumerate}
\end{corollary}
\begin{corollary} The following statements are equivalent.
\begin{enumerate}
\item[(1)] $M$ is locally  umbilic.
\item[(2)] $M$ is locally contained in the intersection of a Hyperbolic (or de Sitter) with a hyperplane.
\item[(3)] $M$ locally admits only one bi-normal field ${\rm{\bf B}} $ and $\nu$-umbilic (not $\nu$-flat), for some normal field $\nu.$
\end{enumerate}
\end{corollary}
\begin{corollary}
If $M$ admits only one  bi-normal field ${\rm{\bf B}} $ (not ${\rm{\bf B}} $-flat) then it is not pseudo-umbilic.
\end{corollary}
\begin{remark}\label{equivalent} The results in \cite{F3} are also true for the spacelike surfaces in $\mathbb R_1^4.$ So that the following statements are equivalent:
 \begin{enumerate}
 \item[(1)] $M$ has two everywhere defined orthogonal asymptotic fields,
 \item[(2)] $M$ is pseudo-umbilic,
 \item[(3)] The normal curvature of $M$ vanishes at every point,
 \item[(4)] All points of $M$ are semi-umbilic.
 \end{enumerate}
\end{remark}
\section{Bi-normal Fields on Ruled Spacelike Surface in $\mathbb R_1^4$}
The notion of ruled surface in $\mathbb R^4$ have been introduced by Lane in \cite{lane}. It is similar to ruled (spacelike) surface in $\mathbb R_1^4$ and can be introduced by the similar way. A surface $M$ in $\mathbb R_1^4$ is called ruled if through every point of $M$ there is a straight line that lies on $M.$ We have a local parameterization of $M$
\begin{equation}\label{rulled}{\rm{\bf X}} (u,t)=\alpha(t)+ uW(t),\ t\in I\subset \mathbb R, u\in \mathbb R,\end{equation}
where $\alpha(t)$ is a differential curve in $\mathbb R_1^4$ and $W(t)$ is a smooth vector field along  $\alpha(t).$\\
\indent A ruled surface $M$ is called developable if its Gaussian curvature identifies zero.\\
\indent It is from  ${\rm{\bf X}}_u=W(t), {\rm{\bf X}}_t(0,t)=\alpha'(t)$ and $M$ is spacelike that both $W(t)$ and $\alpha'(t)$ are spacelike.  We can assume that $|W|=|\alpha'|=1$ and $\langle W,\alpha'\rangle=0.$\\
\indent The coefficients of the first fundamental form of $M$ are
$$g_{11}=\langle {{\textbf  X}}_u,{{\textbf  X}}_u\rangle=\langle \alpha',\alpha'\rangle+2t\langle \alpha',W'\rangle+t^2\langle W',W'\rangle,$$
$$g_{12}=\langle {{\textbf  X}}_u,{{\textbf  X}}_t\rangle=0,\ g_{22}=\langle {{\textbf  X}}_t,{{\textbf  X}}_t\rangle=1.$$
Since $M$ is spacelike, $\langle W',W'\rangle>0.$\\
\indent Let ${\rm{\bf n}}$ be a normal field on $M,$ the coefficients of the second fundamental form associated ${\rm{\bf n}}$ are defined as following
$$b_{11}^{\textbf  n}=\langle {{\textbf  X}}_{uu},\textbf  n\rangle=\langle \alpha'',\textbf  n\rangle+t\langle W'',\textbf  n\rangle,\  b_{12}^{\textbf  n}=\langle {{\textbf  X}}_{ut},\textbf  n\rangle=\langle W',\textbf  n\rangle,$$$$ b_{22}^{\textbf  n}=\langle {{\textbf  X}}_{tt},\textbf  n \rangle=0.$$
Therefore,
\begin{equation}\label{An}S^{\textbf  n}=(g_{ij})^{-1}.(b_{ij}^{\textbf n})=\frac{1}{g_{11}}\left[\begin{matrix}b_{11}^{\textbf n}&b_{12}^{\textbf n}\\b_{12}^{\textbf n}g_{11}&0\end{matrix}\right],\ K^{\textbf  n}=-\frac{(b_{12}^{\textbf n})^2}{g_{11}}.\end{equation}
So, \begin{equation}\label{rubinor}K^{\textbf n}=0\ \Leftrightarrow\ b_{12}^{\textbf n}\ \Leftrightarrow\ \langle W',\textbf n\rangle=0.\end{equation}
The following proposition gives us the number of bi-normal directions at each point on a ruled surface.
\begin{proposition}\label{birul} Let $M$ be a ruled spacelike surface given by (\ref{rulled}), we then have:
\begin{enumerate}
\item at the point such that $\left\{\alpha',W,W'\right\}$ is linearly dependent each normal vector is bi-normal direction;
\item at the point such that $\left\{\alpha',W,W'\right\}$ is linearly independent $M$ admits only one bi-normal direction.
\item $M$ is pseudo-umbilic if and only if umbilic.
\end{enumerate}
\end{proposition}
\begin{proof}\
\begin{enumerate}
\item Since
$$\langle \alpha',W\rangle=0,\  \langle W',W\rangle=0   $$
and $\left\{\alpha',W,W'\right\}$ is linearly dependent, $W'\in T_pM.$ Therefore, by using (\ref{rubinor}), we imply that each normal vector on $M$ is bi-normal direction.
\item Since
$$K{\textbf  n}=0\ \Leftrightarrow\ \left\{ \begin{aligned}&\langle \textbf  n,{{\rm{\bf X}}}_u \rangle=0,\\&\langle \textbf  n,{{\rm{\bf X}}}_v\rangle=0,\\ &\langle \textbf  n,W'\rangle=0,  \end{aligned}\right.\ \Leftrightarrow\ \left\{ \begin{aligned}&\langle \textbf  n,\alpha' \rangle=0,\\&\langle \textbf  n,W\rangle=0,\\ &\langle \textbf  n,W'\rangle=0, \end{aligned}\right. $$
$\textbf n$ is an unit bi-normal direction if and only if
$$\textbf n=\frac{\alpha'\wedge W\wedge W'}{|\alpha'\wedge W\wedge W'|}.$$
It is followed from the fact that $\alpha',W,W'$ are spacelike that the unique unit bi-normal direction on $M$ is timelike.
\item Since $M$ admits only one bi-normal field, it is followed Theorem \ref{theoum} that $M$ is pseudo-umbilic iff umbilic.
\end{enumerate}
\end{proof}
\begin{remark}\begin{enumerate}
\item The Proposition \ref{birul} is also true for the ruled surfaces in $\mathbb R^4.$
\item  Using the Gauss equation we can show that the Gaussian curvature of a spacelike surface in $\mathbb R_1^4$ can be defined by sum of $K^{e_1}$  and $K^{e_2},$ where $\{e_1,e_2 \}$ is an orthogonal frame of normal bundle of surface. Therefore, a ruled spacelike surface is developable iff $\left\{\alpha',W,W'\right\}$ is linearly dependent.
\item Similarly the results on the surfaces in $\mathbb R^4$ (see \cite{lit}), it is easy to show that if a spacelike surface $M$ is planar and the  causal character of its ellipse curvature (see \cite{izu1}) is invariant then $M$ is a ruled developable surface.
\end{enumerate}
\end{remark}
Lane \cite{lane} showed that if a ruled surface in $\mathbb R^4$ is minimal then it is contained in a hyperplane and of course it is either plane or helicoid. We have the same results for the maximal ruled  spacelike surfaces in $\mathbb R_1^4.$ That means a ruled spacelike surface in $\mathbb R_1^4$ is maximal if and only if it is maximal in a spacelike hyperplane.
\section{Bi-normal Fields on Spacelike Surfaces of Revolution}
Let $C$ be a spacelike  curve in $\text{span}\{e_1,e_2,e_4\}$ parametrized by  arc-length,
$$z(u)=\left(f(u),g(u),0,\rho(u)\right),\ \rho(u)>0,\ \ u\in I.$$
The orbit of $C$ under the action of the orthogonal transformations of $\mathbb R_1^4$ leaving the spacelike plane $Oxy,$
$$A_S=\left[\begin{matrix}1&0&0&0\\0&1&0&0\\0&0&\cosh v&\sinh v\\0&0&\sinh v&\cosh v\end{matrix}\right],\ v\in\mathbb R,$$
is a  surface given by
\begin{equation}\label{hr1} {\rm[RH]}\qquad
{\rm{\bf X}}(u,v)=\left(f(u),g(u),\rho(u)\sinh v, \rho(u)\cosh v\right),\ u\in I,\ v\in\mathbb R.
\end{equation}
The coefficients of the first fundamental form of [RH] are
$$g_{11}=(f'(u))^2+(g'(u))^2-(\rho'(u))^2=1,\  g_{12}=0,\ g_{22}=\left(\rho(u)\right)^2>0. $$
It follows that [RH] is a spacelike surface, which is called  the {\it spacelike surface of revolution of hyperbolic type} in $\mathbb R_1^4$. From now on we always assume that $f'\ne0,g'\ne0$ and $\rho'\ne 0$.
\begin{proposition} Suppose that $f'g''-f''g'\ne0,$ we then have:
\begin{enumerate}
\item[(a)] [RH] admits exactly two bi-normal fields and its asymptotic fields  are  orthogonal,
\item[(b)] There exists only one normal field $\nu$ satisfying that [RH] is $\nu$-umbilic.
\end{enumerate}
\end{proposition}
\begin{proof}
(a) Let $\textbf  n=(n_1,n_2,n_3,n_4)$ be a normal field on $M$, we have
$$\langle {\rm{\bf X}}_u,\textbf  n\rangle=0,\ \langle {\rm{\bf X}}_v,\textbf  n\rangle=0.$$
That means
\begin{equation}\label{b12=0}
n_1f'+n_2g'+n_3\rho'\sinh v-n_4\rho'\cosh v=0,\ \rho(n_3\cosh v-n_4\sinh v)=0.
\end{equation}
Since (\ref{b12=0}),
$$b_{12}^{\textbf  n}=\langle {{\rm{\bf X}}}_{uv},\textbf  n\rangle=\rho'\left(n_3\cosh v-n_4\sinh v\right)=0.$$
Therefore,
\begin{equation}\label{AofR}
\det\left(S^{\textbf  n}\right)=\frac{b_{11}^{\textbf  n}.b_{22}^{\textbf  n}}{\rho^2},
\end{equation}
where $b_{ij}^{\textbf n}$ are the coefficients of the second fundamental form associated with $\textbf n$ of [RH].\\
\indent On the other hand we have
$$\langle {{\rm{\bf X}}}_u,{{\rm{\bf X}}}_u\rangle=1\ \Rightarrow\ \langle {{\rm{\bf X}}}_{uu},{{\rm{\bf X}}}_u\rangle=0,  $$
$$\langle {{\rm{\bf X}}}_u,{{\rm{\bf X}}}_v\rangle=0\ \Rightarrow\ \langle {{\rm{\bf X}}}_{uu},{{\rm{\bf X}}}_v\rangle=-\langle {{\rm{\bf X}}}_{u},{{\rm{\bf X}}}_{uv}\rangle=0.   $$
So, $\{{{\rm{\bf X}}}_u,{{\rm{\bf X}}}_v,{{\rm{\bf X}}}_{uu}\}$ is linearly independent. Therefore, $b_{11}^{\textbf  n}=0$ if and only if $\textbf  n$ is parallel to
$${\rm{\bf B}}_1={{\rm{\bf X}}}_u\wedge {{\rm{\bf X}}}_v\wedge {{\rm{\bf X}}}_{uu}.$$
It is easy to show that $b_{22}^{\textbf  n}=0$ if and only if $\textbf  n$ is parallel to ${{\rm{\bf B}}_2}=(-g',f',0,0).$ ${\rm{\bf X}}_v $ then is asymptotic field associated with ${\rm{\bf B}}_2 .$\\
\indent Since $f'g''-f''g'\ne0,$ ${\rm{\bf B}}_1,{\rm{\bf B}}_2$ are linearly independent. So,  [RH] admits exactly two bi-normal fields. \\
\indent (b) Using base $\{{\rm{\bf X}}_u,{\rm{\bf X}}_v\}$ for tangent planes of [RH], we have
$$S^{{\rm{\bf B}}_1}=\left[\begin{matrix}0&0\\0&-f'g''+f''g'\end{matrix}\right],\ S^{{\rm{\bf B}}_2}=\left[\begin{matrix}f'g''-f''g'&0\\0&0\end{matrix}\right].$$
Therefore, [RH] is $\nu$-umbilic, where $\nu={{\rm{\bf B}}_1}-{{\rm{\bf B}}_2}.$ Remark \ref{equivalent} shows that the normal curvature of [RH] identifies zero, [RH] has two orthogonal asymptotic fields everywhere, and [RH] is semi-umbilic.
\end{proof}
\begin{remark}\
\begin{enumerate}
\item[(a)] If $f'g''-f''g'=0$ then $M$ is contained in a hyperplane.
\item[(b)] It is similar to the spacelike surfaces of revolution of elliptic type.
\end{enumerate}
\end{remark}
\section{Bi-normal Fields on The Rotational Spacelike Surfaces Whose Meridians Lie in Two-dimension planes}
 Romero Fuster et. al \cite{F2} showed that there always an open region of a generic, compact $2$-manifold in $\mathbb R^4$ all whose points admit at least one bi-normal direction and at most n of them. This result is not true in the general case. This section gives a class of spacelike surfaces whose points can admit non, one, two or infine bi-normal directions. It is similar to them on $\mathbb R^4.$\\
\indent Let $C$ be a spacelike curve contained in $\text{span}\{e_1,e_3\}$ and parametrized by
$$r(u)=\left(f(u),0,g(u),0\right),\ u\in I,$$
and
$$A=\left[\begin{matrix}\cos\alpha v&-\sin\alpha v&0&0\\\sin\alpha v&\cos\alpha v&0&0\\0&0&\cosh\beta v&\sinh\beta v\\0&0&\sinh\beta v&\cosh\beta v\end{matrix}\right],\ v\in[0,2\pi) ,$$
such that
$$\alpha^2f^2(u)-\beta^2g^2(u)>0,$$
be a orthogonal transformations of $\mathbb R_1^4$, where  $u\in J\subset \mathbb R$ and $\alpha,\beta$ are positive constants.\\
\indent The orbit of $C$ under the action of the orthogonal transformations $A$
is a  surface [RS] given by
\begin{equation}\label{gesurrevoone}
{\rm{\bf X}} (u,v)=\left(f(u)
\cos\alpha v,f(u)\sin\alpha v,g(u)\cosh\beta v,g(u)\sinh\beta v\right).
\end{equation}
The coefficients of the first fundamental form of [RS] are
$$g_{11}=(f')^2+(g')^2>0,\ g_{12}=0,\ g_{22}=\alpha^2f^2-\beta^2g^2>0.$$
That means [RS] is  spacelike. It is called {\it  rotational spacelike surface whose meridians lie in two-dimension planes of type I}.\\
Choosing $\{{\rm{\bf n}}_1,{\rm{\bf n}}_2\} $ is an orthonormal frame field on [RS], where
$${\rm{\bf n}}_1=\frac{1}{\sqrt{(f')^2+(g')^2}}\left(g'\cos\alpha v,g'\sin\alpha v,-f'\cosh\beta v,-f'\sinh\beta v\right),  $$
$${\rm{\bf n}}_2=\frac{1}{\sqrt{\alpha^2f^2-\beta^2g^2}}\left(-\beta g\sin\alpha v,\beta g\cos\alpha v,\alpha f\sinh \beta v,\alpha f\cosh\beta v\right),$$
then the coefficients of the second fundamental form associated to ${\rm{\bf n}}_1$ and ${\rm{\bf n}}_2$ are defined by
$$b_{11}^{{\rm{\bf n}}_1}=\frac{f''g'-f'g''}{\sqrt{(f')^2+(g')^2}},\  b_{12}^{{\rm{\bf n}}_1}=0,\ b_{22}^{{\rm{\bf n}}_1}=-\frac{\beta^2f'g+\alpha^2fg'}{\sqrt{(f')^2+(g')^2}},$$
$$b_{11}^{{\rm{\bf n}}_2}=0,\  b_{12}^{{\rm{\bf n}}_2}=\frac{\alpha\beta(f'g-fg')}{\sqrt{\alpha^2f^2-\beta^2g^2}} ,\ b_{22}^{{\rm{\bf n}}_2}=0,$$
 respectively.\\
Let ${\rm{\bf B}}$ be a normal field on [RS], we have
$${\rm{\bf B}}=\lambda {\rm{\bf n}}_1+\mu {\rm{\bf n}}_2,$$
where $\lambda,\mu$ are smooth functions on [RS]. Then we have
$$(b_{ij}^{\rm{\bf B}})=\lambda (b_{ij}^{{\rm{\bf n}}_1})+\mu (b_{ij}^{{\rm{\bf n}}_2})=\left[\begin{matrix}\lambda b_{11}^{{\rm{\bf n}}_1}&\mu b_{12}^{{\rm{\bf n}}_2}\\\mu b_{12}^{{\rm{\bf n}}_2}&\lambda b_{22}^{{\rm{\bf n}}_1}\end{matrix}\right].$$
So,
$$K^{\rm{\bf B}}=\frac{\lambda^2b_{11}^{{\rm{\bf n}}_1}.b_{22}^{{\rm{\bf n}}_1}-\mu^2(b_{12}^{{\rm{\bf n}}_2})^2}{\left((f')^2+(g')^2\right)\left(\alpha^2f^2-\beta^2g^2\right)} .$$
Therefore,
\begin{enumerate}
\item[(a)] $\textbf  n_2$ is bi-normal if and only if $f=cg$, where $c$ is constant satisfying $\alpha^2-c\beta^2>0$. Then $b_{22}^{\textbf  n_1}=0$. So, $\textbf  n_1$ is also bi-normal. And it is easy to show that [RS] is a planar. That means [RS] is planar if and only if $C$ is a line passing through the origin.
\item[(b)] $\textbf  n_1$ is bi-normal if and only if either
\begin{equation}\label{onebi}\begin{aligned}&f=cg+c_1\ \text{or}\  \alpha^2fg'+\beta^2f'g=0,\end{aligned}\end{equation}
where $c,c_1$ are constant. In this case, if $c_1\ne0$ then [RS] admits only one bi-normal field that is $\textbf  n_1.$ Therefore, [RS] admits only one bi-normal field if and only $\textbf n_1 $ is bi-normal and $\textbf  n_2$ is not bi-normal. Which takes place if and only if (\ref{onebi}) is true and $c_1\ne0.$ For example
$${\rm{\bf X}}(u,v)=\left(u\cos v,u\sin v,\cosh v,\sinh v\right),\ u>1,\ v\in(0,2\pi).$$
\item[(c)] [RS] does not admit any bi-normal field if and only if
$$-(f''g'-f'g'')(\beta^2f'g+\alpha^2fg')<0\ \text{and}\ \alpha\beta(f'g-g'f)\ne0.$$
For example
$${\rm{\bf X}}(u,v)=\left(u^2\cos v,u^2\sin v,u\cosh v,u\sinh v\right),\ u>1,\ v\in(0,2\pi).$$
\item[(d)] [RS] admits exactly two bi-normal fields if and only if
$$-(f''g'-f'g'')(\beta^2f'g+\alpha^2fg')>0\ \text{and}\ \alpha\beta(f'g-g'f)\ne0.$$
For example
$${\rm{\bf X}}(u,v)=\left(e^{2u}\cos v,e^{2u}\sin v,e^{-u}\cosh v,e^{-u}\sinh v\right),\ u>1,\ v\in(0,2\pi).$$
\end{enumerate}
It is similar to the rotational spacelike surfaces whose meridians lie in two-dimension planes of type II. This result is also true for the rotational surfaces whose meridians lie in two-dimension planes in $\mathbb R^4.$

\bigskip

\address{ Dang Van Cuong\\
Department of Natural Sciences \\
Duy Tan University \\
Danang \\
Vietnam
}
{dvcuong@duytan.edu.vn}
\end{document}